\documentclass[twoside, 12pt]{article}
\usepackage{amssymb, amsmath, mathrsfs, amsthm}%, wasysym}
\usepackage{graphicx}
\usepackage{color}
\usepackage[top=2cm, bottom=2cm, left=2cm, right=2cm]{geometry}
\usepackage{float, caption, subcaption}

\input{epsf}

\newcommand{\bd}{\begin{description}}
\newcommand{\ed}{\end{description}}
\newcommand{\bi}{\begin{itemize}}
\newcommand{\ei}{\end{itemize}}
\newcommand{\be}{\begin{enumerate}}
\newcommand{\ee}{\end{enumerate}}
\newcommand{\beq}{\begin{equation}}
\newcommand{\eeq}{\end{equation}}
\newcommand{\beqs}{\begin{eqnarray*}}
\newcommand{\eeqs}{\end{eqnarray*}}

\definecolor{DarkGreen}{rgb}{0.2, 0.6, 0.3}

\newtheorem{theorem}{Theorem}[section]

\newtheorem{proposition}{Proposition}[section]

\newtheorem{definition}{Definition}
\newtheorem{corollary}[theorem]{Corollary}

\newtheorem{claim}{Claim}
\newtheorem{fact}{Fact}

\newtheorem{problem}{Problem}[section]

\begin{document}
\title{\textbf{On the set-coloring Ramsey numbers of graphs\footnote{Supported by the National Science Foundation of China
(Nos. 12471329 and 12061059) and the Qinghai Key Laboratory
of Internet of Things Project (2017-ZJ-Y21).}}}

\author{Mengya He \footnote{School of Mathematics and Statistis, Qinghai
Normal University, Xining, Qinghai 810008, China. {\tt
hmy8536@163.com}}, \ \ Yaping Mao\footnote{Corresponding author: Academy of Plateau Science and Sustainability,
and School of Mathematics and Statistics, Xining, Qinghai 810008,
China. {\tt yapingmao@outlook.com} }}
\date{}
\maketitle

\begin{abstract}
The \textit{set-coloring Ramsey number} $\mathrm{R}_{r, s}(G_1,G_2,...,G_r)$ is the least $n \in \mathbb{N}$ such that every coloring $\chi: E\left(K_n\right) \rightarrow\binom{[r]}{s}$ contains a monochromatic copy of $G_i$, that is, a color $i \in[r]$ such that $i \in \chi(e)$ for every $e \in E(G_i)$. If $G_1=G_2=\cdots=G_r=G$, then we write $\mathrm{R}_{r,s}(G)$ for short. 
In 2022, Le asked to find lower and upper bounds for $\mathrm{R}_{s, t}(G)$ with various kinds of graphs $G$ such as stars, paths, cycles, etc.  
In this paper, we obtain exact values or bounds for the set-coloring Ramsey numbers of stars, paths, matchings, etc. By Lov\'{a}sz Local Lemma, we give a lower bound for the set-coloring Ramsey number for general graphs.\\[2mm]
{\bf Keywords:} Coloring; Ramsey Theory; Set-coloring Ramsey number; Path; Star.  \\[2mm]
{\bf AMS subject classification 2020:} 05C15; 05C30; 05C55.
\end{abstract}

\section{Introduction}

In this paper, all graphs are described as undirected, finite, and simple, while any undefined concepts or notations are referenced in \cite{BM08}. Let $G(V,E)$ be a simple graph with vertex set $V$ and edge set $E$. Let $P_n$ be the path with $n$ vertices and $n-1$ edges and  $C_n$ be the cycle with $n$ vertices and $n$ edges. For any $v\in V(G)$, let $N_{G}(v)$ be the neighbors of $v$ and $d_{G}(v)=|N_{G}(v)|$. For any vertices $u,v\in G$, $d_{G}(u,v)$ denotes the distance of $u,v$ in $G$, that is the length of a shortest $uv$-path. If there is no path connecting $u$ and $v$ (that is, if $u$ and $v$ lie in distinct components of $G$), we set $d_{G}(u,v)=\infty$. If $S_1,S_2\subset V(G)$, let $E[S_1,S_2]$ be the edge between $S_1$ and $S_2$. A \textit{matching} in a graph is a set of pairwise nonadjacent edges. If $M$ is a matching, the two ends of each edge of $M$ are said to be \textit{matched under $M$}, and each vertex incident with an edge of $M$ is said to be \textit{covered by $M$}. A \textit{perfect matching} is one which covers every vertex of the graph, a maximum matching one which covers as many vertices as possible. 

\begin{definition}
The \textit{$r$-color Ramsey number} $\mathrm{R}_r(G_1,G_2,...,G_r)$ is defined to be the minimum $n \in \mathbb{N}$ such that every $r$-coloring $\chi: E\left(K_n\right) \rightarrow\{1, \ldots, r\}$ of the edges of the complete graph on $n$ vertices contains a monochromatic copy of $G_i$ with color $i$.
\end{definition}
If $G_1=\cdots=G_r=K_{k}$, then $\mathrm{R}_r(G_1,G_2,...,G_r)=\mathrm{R}_r(k)$.
Known results of small Ramsey numbers can be found in a dynamic survey by Radziszowski\cite{Ra94}. We also refer
to Graham, Rothschild and Spencer \cite{GRS80} for an overview of Ramsey theory.

For any positive integers $r,s$ with $r\geq s$, we call $\chi$ a \emph{$(r,s)$-coloring} of $G$ if $\chi$ is an edge-coloring of $G$ where each edge is colored with $s$ distinct colors from $[r]$. Under this coloring, the set $\{1,2,...,r\}$ is called a \textit{color set}, and each $i \ (1\leq i\leq s)$ is called a \textit{color element}. Let $\chi$ be a $(r,s)$-coloring of $G$. A subgraph $G_1\subseteq G$ is \emph{monochromatic} with color element $i$ if there is
some color $i\in [r]$ such that $i\in \chi(e)$ for all edges $e\in E[G_1]$. 

\begin{definition}
The \textit{set-coloring Ramsey number} $\mathrm{R}_{r, s}(k)$ is the least $n \in \mathbb{N}$ such that every coloring $\chi: E\left(K_n\right) \rightarrow\binom{[r]}{s}$ contains a monochromatic clique of size $k$, that is, a set $S \subset V\left(K_n\right)$ with $|S|=k$ and a color $i \in[r]$ such that $i \in \chi(e)$ for every $e \in\binom{S}{2}$.
\end{definition}

If $s=1$, then $\mathrm{R}_{r, s}(k)=\mathrm{R}_r(k)$. The exploration of set-coloring Ramsey numbers began in the 1960s through the work of Erd\H{o}s, Hajnal and Rado \cite{EHR65}. They proposed the conjecture that $\mathrm{R}_{r, r-1}(k) \leqslant 2^{\delta(r) k}$ for some function $\delta(r) \rightarrow 0$ as $r \rightarrow \infty$. This conjecture was confirmed by Erd\H{o}s and Szemer\'{e}di \cite{ES72} in 1972, they showed that 
$$
2^{\Omega(k / r)} \leqslant \mathrm{R}_{r, r-1}(k) \leqslant r^{O(k / r)} .
$$
The lower bound is established through a straightforward random coloring approach, while the upper bound was demonstrated by Erd\H{o}s and Szemer\'{e}di. They proved that in any 2-coloring where one color has a density of at most 
$1 / r$, there exists a monochromatic clique of size $c(r) \log n$, where $c(r)=\Omega(r / \log r)$. This finding was then utilized concerning the color assigned to the majority of edges. In the case of more general values of $s$, significant advancements were made only recently by by Conlon, Fox, He, Mubayi, Suk and Verstraëte \cite{CFMSV24}, who demonstrated the bound for $\mathrm{R}_{r, s}(k)$. For the lower bound they proved

$$
\mathrm{R}_{r, s}(k) \leqslant \exp \left(\frac{c k(r-s)^2}{r} \log \frac{r}{\min \{s, r-s\}}\right)
$$
where $c$ is a positive constant.
Also by connecting the relation between error-correcting codes and set-Ramsey number, they show that
$$
\mathrm{R}_{r, s}(k) \geqslant \exp \left(\frac{c^{\prime} k(r-s)^3}{r^2}\right)
$$
where $c'$ is a positive constant. They also observed that a straightforward random coloring provides a lower bound of $\mathrm{R}_{r, s}(k)$ which is stronger when $r-s \ll \sqrt{r}$. That is $\mathrm{R}_{r, s}(k) \geqslant 2^{\Omega(k(r-s) / r)}$. In 2022, Le \cite{Le22} gave the following lower bound.
\begin{theorem}\upshape{\cite{Le22}}
For all positive integers $r,s,k$ such that $r>s,k\geq 3$, we have 
$$
\mathrm{R}_{r,s}(k)> \left\lfloor\frac{k}{e}r^{-1/k}\left(\frac{r}{s}\right)^{(k-1)/2}   \right\rfloor.
$$
\end{theorem}

The set-coloring Ramsey number has important applications on error-correcting codes; see \cite{CFMSV24,CFPZ24}. A natural extension of Ramsey numbers, initially proposed by Erd\H{o}s, Hajnal and Rado in 1965 \cite{EHR65}, presents the following question. For a given graph $H$ and integers $r \geq s \geq 1$, determine the minimum number $N$ such that in any $r$-colored complete graph on $N$ vertices there is a copy of $H$ whose edges are colored using at most $s$ different colors. We denote this quantity as 
by $\mathrm{R}_s^r(H)$. It is worth noting that when $s=1$, the problem simplifies to the classical Ramsey problem, which is known to be the most challenging instance in this context.

Buci\'{c} and Khamseh \cite{BK23} gave the definition of $(r-1)$-chromatic Ramsey numbers for general graphs.
\begin{definition}\upshape{\cite{BK23}}
The \textit{$(r-1)$-chromatic Ramsey number}, denoted by $\mathrm{R}_{r-1}^r(G_1, G_2, \ldots, G_r)$, is defined to be the least number $n \in \mathbb{N}$ such that in any $r$-coloring of the complete graph $K_n$, for some $i$ we can find a copy of $G_i$ whose edges do not use color $i$.    
\end{definition}

Motivated by the success of set-coloring Ramsey numbers for complete graphs and $(r-1)$-chromatic Ramsey numbers for general graphs, we study the set-coloring Ramsey numbers for general graphs in this paper. 
\begin{definition}
The \textit{set-coloring Ramsey number} $\mathrm{R}_{r, s}(G_1,G_2,...,G_r)$ is the least $n \in \mathbb{N}$ such that every coloring $\chi: E\left(K_n\right) \rightarrow\binom{[r]}{s}$ contains a monochromatic copy of $G_i$ with color element $i$, that is, a color $i \in[r]$ such that $i \in \chi(e)$ for every $e \in E(G_i)$.
\end{definition}

If $G_1=\cdots=G_r=G$, then we write $\mathrm{R}_{r,s}(G)$ for short. If $G=K_k$, then we write $\mathrm{R}_{r,s}(k)$ for short, which was studied by Le in \cite{Le22}.
Note that $\mathrm{R}_{r, s}(G_1,\ldots,G_r)\leq \mathrm{R}_{r}(G_1,\ldots,G_r)$ and if $s=1$, then $\mathrm{R}_{r, s}(G_1,\ldots,G_r)=\mathrm{R}_{r}(G_1,\ldots,G_r)$. Observe that $\mathrm{R}_{r, r-1}(G_1,\ldots,G_r)\leq \mathrm{R}_{r-1}^r(G_1,\ldots, G_r)$. 

Le \cite{Le22} showed that $\mathrm{R}_{3,2}\left(C_4\right)=5$ and asked the following problems.
\begin{problem}
Find lower and upper bounds for $\mathrm{R}_{s, t}(G)$ with various kinds of graphs $G$ such as stars, paths, cycles, etc.    
\end{problem}

Motivated by the problem, we study the set-coloring Ramsey numbers in this paper. Our main results are as follows. 
\begin{itemize}
    \item For the paths, we show that $\mathrm{R}_{3, 2} (P_{3},P_{n},P_{n})=\mathrm{R}_{3, 2}(K_{3},P_{n},P_{n})=n$ (Theorem \ref{them:A-P2-Pn}). For the general $r,s$, we prove that 
$\mathrm{R}_{r,s} (P_{n_1},\ldots,P_{n_r})\leq \sum^{r}_{i=1}\frac{n_i-2}{2s}+\frac{1}{2}$ where $r,s$ are two positive integers with $r>s$ and $n_i\geq 2$ (Corollary \ref{cor-PS}). Moreover, we show that $\mathrm{R}_{r, s} (P_{n_1},\ldots,P_{n_r})\geq \left\lfloor\frac{n_1-1}{2}\right\rfloor \left\lfloor\frac{r}{s}\right\rfloor+1$ where $r,s$ be two positive integers with $r>s$ and $2\leq n_i\leq n_j$ and  $1\leq i\leq j\leq r$ (Theorem \ref{th-PL}). If $s|r$ and $n\geq 2$ is even, we prove that $\mathrm{R}_{r, s} (P_{n})=\frac{r(n-2)}{2s}$ (Corollary \ref{cor:A-Pn-Pn}). 

\item  For the stars, we prove that if $\lfloor \frac{rn_1-r}{s}\rfloor\leq {r\choose s}$, then $\left\lfloor (rn_1-r)/s \right\rfloor+1\leq  \mathrm{R}_{r, s} (K_{1,n_1},\ldots,K_{1,n_r})\leq \left\lceil (\sum_{i=1}^rn_i-r+1)/s)\right\rceil+1$ where $n_1,\ldots, n_r,r,s$ be positive integers with $3\leq n_1\leq \cdots\leq n_r$ and $r>s$ (Theorem \ref{S1}).
Moreover, the upper bound is sharp. We get $\mathrm{R}_{3, 2} (K_{1,2},K_{1,n},K_{1,n})=n+1$ for $n\geq 2$ to show the sharpness (Theorem \ref{S2}). 

\item By Lov\'{a}sz Local Lemma, we prove that for any positive integers $r,s$ where $r>s$ and a connected graph $G$ with $n\geq 4$ vertices and $m$ edges such that ${n-1\choose 2}\leq m \leq {n\choose 2}$, we have $\mathrm{R}_{r,s}(G)\geq (\frac{m-1}{c_{2}sr\ln ((m-1)/c_2)})^{sr}$ (Theorem \ref{L1}), where $c_1,c_2,c_3$ are three positive numbers with $c_1\leq \frac{r-s}{r}$, and $c_3-c_{1}c_{2}+3n<0$. 
 Also we gave a lower bound, that is, for any positive integers $r,s$ with $r>s$ and a graph $G_i$ with $n_i$ vertices and $m_i$ edges, we have  
$\mathrm{R}_{r,s}(G_1,\ldots,G_r)\geq \frac{n^{n/n'}}{e}(\left(\frac{r}{s}\right)^{m}\frac{(2m)^{m}}{\sum^{r}_{i=1}(et^{2}_i)^{m_i}})^{1/n'}$,
where $m=\min\{m_1,\ldots,m_r\}$, $n=\min\{n_1,\ldots,n_r\}$, and $n'=\max\{n_1,\ldots,n_r\}$ (Theorem \ref{L2}). 
\end{itemize}

\section{Results for special graphs}

The following result is well-known, which will be used later.
\begin{theorem}\upshape{\cite{H94}}\label{KM}
$(i)$  The complete graph $K_{2n}$ can be decomposed into $(2n-1)$ edge-disjoint perfect matchings.

$(ii)$  The complete graph $K_{2n+1}$ can be decomposed into $n$ edge-disjoint Hamilton cycles.
\end{theorem}

\begin{theorem}\label{KM-odd}
The complete graph $K_{2n+1}$ can be decomposed into $2n+1$ edge-disjoint maximal matchings that each matching not cover only one vertex of $V(K_{2n+1})$.
\end{theorem}
\begin{proof}
By the proof Theorem \ref{KM}, $K_{2n+1}$ can be decomposed into $n$ edge-disjoint Hamilton cycles, 
$$
C_i=v_{2n+1}v_{i}v_{i-1}v_{i+1}v_{i-2}\cdots v_{i+n-1}v_{i-n}v_{2n+1},
$$
where $1\leq i\leq n$ and except $v_{2n+1}$ all subscripts are taken as the integers $1,2,\ldots,2n~(\bmod~ 2n)$. Let $M^{1}=\{v_{i+n-1}v_{i-n}\,|\, 1\leq i\leq n~and~i~is~even\}$ and $M^{2}=\{v_{i}v_{i-1}\,|\, 1<i\leq n~and~i~is~odd \}$. Note that if $1\leq i\leq n$ and $i$ is even, then $v_{i+n-1}v_{i-n}\in E(C_i)$ and if  $1<i\leq n$ and $i$ is odd, then $v_{i}v_{i-1}\in C_i$. Let $M_{2n+1}=M^1 \cup M^2\cup \{v_{2n+1}v_{1}\}$ where $v_{2n+1}v_{1}\in E(C_1)$. Then $M_{2n+1}$ is the maximal matching of $K_{2n+1}$ and if $n$ is even, then $v_n$ is the only not covered vertex by $M_{2n+1}$; if $n$ is odd, then $v_{2n}$ is the only not covered vertex by $M_{2n+1}$.

For each $e\in M_{2n+1}$, there is a $C_j(1\leq j\leq n)$ such that $e\in C_j$. Since $C_{j}-e=P_{2n+1}$, it follows that $C_{j}-e$ can be decomposed into $2$ matchings $M^{1}_{j}$ and $M^{2}_{j}$ where both $M^{1}_{j}$ and $M^{2}_{j}$ are the maximal matching of $K_{2n+1}$ and the only not covered vertex by $M^{1}_{j}$ or $M^{2}_{j}$ is one of the incident vertex of $e$. Then $\bigcup^{n}_{i=1}C_i=\bigcup^{n}_{i=1}(M^{1}_{i}\cup M^{2}_{i})\cup M_{2n+1}$. 
\end{proof}

\subsection{Paths}

In this section, we consider the $(3,2)$-coloring for paths.
\begin{theorem}\label{them:A-P2-Pn}
For $n\geq 2$, we have 
$$
\mathrm{R}_{3, 2} (P_{3},P_{n},P_{n})=\mathrm{R}_{3, 2}(K_{3},P_{n},P_{n})=n.
$$
\end{theorem}

\begin{proof}
To show $\mathrm{R}_{3, 2} (P_{3},P_{n},P_{n})\geq n$, we construct a $(3,2)$-coloring of $K_{n-1}$ with color element set $\{1,2,3\}$ such that each edge of $K_{n-1}$ receives the color $\{2,3\}$. Then there is no monochromatic copy of $P_3$ with element $1$. Observe that there is no $P_n$ in $K_{n-1}$. Of course, there is no monochromatic copy of $P_n$ with color containing $2$ or $3$. This shows that $\mathrm{R}_{3, 2} (K_{3},P_{n},P_{n})\geq \mathrm{R}_{3, 2} (P_{3},P_{n},P_{n})\geq n$.

We prove that $\mathrm{R}_{3, 2}(K_{3},P_{n},P_{n})\leq n$ by induction on $n$. It suffices to show that for any $(3,2)$-coloring $\chi$ of $K_n$, there exists a monochromatic copy of $K_3$ with a color containing $1$ or a monochromatic copy of $P_n$ with a color containing the color element $i$ ($i=2$ or $3$). 
If $n=3$, let $V(K_3)=\{v_1,v_2,v_3\}$, then by pigeonhole principle any two edges of $K_3$ receive some common element in $\{1,2,3\}$. If two edges $v_1v_2,v_2v_3$ have the common element $2$ or $3$, then there is a path $v_1v_2v_3$ containing the element $2$ or $3$. Suppose that the common element is $1$. In order to avoid a monochromatic copy of $K_3$ with the element $1$, $\chi(v_1v_3)=\{2,3\}$. Since the color of $v_1v_2$ or $v_2v_3$ containing the element $2$ or $3$, it follows that there is a monochromatic path $v_1v_2v_3$ containing color $2$ or $3$. Then the result follows for $n=3$.

Assume that the conclusion holds for all integers less than $n$. This means $\mathrm{R}_{3, 2}(K_{3},P_{n-1},P_{n-1})\leq n-1$. We now consider $\mathrm{R}_{3, 2}(K_{3},P_{n},P_{n})$. We assume that for any $(3,2)$-coloring $\chi$ of $K_n$, there is neither a monochromatic copy of $K_3$ with a color containing $1$ nor a monochromatic copy of $P_n$ with a color containing the color element $i$ ($i=2$ or $3$). Let $S$ be a vertex set with $|S|=n-1$ and $v_{n}=V(K_{n})\setminus S$. By induction, for any $(3,2)$-coloring of the edges of $K_{n-1}$, there is a monochromatic copy of $K_3$ with the color containing the element $1$ or a monochromatic copy of $P_{n-1}$ with the color containing the element $2$ or a monochromatic copy of $P_{n-1}$ with the color containing the element $3$. If there is a monochromatic copy of $K_3$ with a color containing the element $1$, then we get a contradiction. Without loss of generality, we  assume that there is a monochromatic copy of $P_{n-1}$ with a color containing the element $2$ and $P_{n-1}=v_1v_2\ldots v_{n-1}$.

If $2\in \chi(v_1v_n)$ or $2\in \chi(v_{n-1}v_n)$, then there is a monochromatic copy of $P_{n}=v_1v_2\ldots,v_n$ with a color containing the element $2$, a contradiction. So $\chi(v_1v_n)=\chi(v_{n-1}v_n)=\{1,3\}$. By the assumption that there is no a monochromatic copy of $K_3$ with the color containing $1$, then $\chi(v_1v_{n-1})=\{2,3\}$.

Then we have the following claim.
\begin{claim}\label{claim 1}
 For any $v_iv_{j} \in K_n[S]$ with $1\leq i<j\leq n-1$, if $\chi(v_iv_{j})=\{1,2\}$, then $3\in \chi(v_iv_n)$ and $3\in \chi(v_jv_n)$.
\end{claim}
\begin{proof}
For any $v_iv_{j} \in K_n[S]$, suppose that $3\notin \chi(v_iv_n)$ or $3\notin \chi(v_jv_n)$. It is impossible the $3\notin \chi(v_iv_n)$ and $3\notin \chi(v_jv_n)$, otherwise $\chi(v_iv_n)=\{1,2\}$ and $\chi(v_jv_n)=\{1,2\}$, we have a monochromatic copy of $K_3$ with the color containing the element $1$, it is a contradiction. Without loss of generality, we assume $3\notin \chi(v_iv_n)$ and $3\in \chi(v_jv_n)$, that means $\chi(v_iv_n)=\{1,2\}$ and $\chi(v_jv_n)=\{2,3\}$. If $j=i+1$, then we can get a monochromatic copy of $P_{n}=v_1\ldots v_iv_nv_{i+1}v_{i+2}\ldots v_{n-1}$ with a color containing $2$, a contradiction. So we can assume $j>i+1$. If $i=1$, then we can get a monochromatic copy of $P_{n}=v_{j-1}v_{j-2}\ldots v_1v_nv_{j}v_{j+1}\ldots v_{n-1}$ with a color containing $2$, a contradiction. If $i\neq 1$, then we also can get a monochromatic copy of $P_{n}=v_{j-1}v_{j-2}\ldots v_{i}v_nv_{j}v_{j+1}\ldots v_{n-1}v_1v_2\ldots v_{i-1}$ with a color containing $2$, a contradiction. Hence $3\in \chi(v_iv_n)$ and $3\in \chi(v_jv_n)$.
\end{proof}

Let $P^{1}$ be the longest path with two end-vertices, say $v^{1}_{1}, v^{1}_{2}$, such that each edge in $P^1$ is with color $\{1,2\}$ in $K_{n}[S]$. 
If $v^{1}_{1}=v^{1}_{2}$, then all the edges of $K_{n}[S]$ receive the colors $\{1,3\}$ and $\{2,3\}$, and hence $v_1v_2\ldots v_{n-1}$ is a monochromatic path $P_{n-1}$ with a color containing the element $3$. Since $\chi(v_1v_n)=\chi(v_{n-1}v_n)=\{1,3\}$, then we have a monochromatic copy of $P_{n}$ with a color containing the element $3$, a contradiction.
Let $P^{2}$ be the longest path with two end-vertices, say $v^{2}_{1}, v^{2}_{2}$, such that each edge in $P^2$ is with color $\{1,2\}$ in $K_{n}[S-V(P^{1})]$. Note that if $v^{2}_{1}=v^{2}_{2}$, then the edges of $K_{n}[S-V(P^{1})]$ receive the colors $\{2,3\}$ and $\{1,3\}$.

Continue this process, we find some paths, say $P^1,P^2,\ldots,P^t$, until $S-\cup^{t}_{i=1}V(P^i)=\emptyset$. We can assume that the length of each path $P_{i} \ (1\leq i\leq t_1)$ is at least $1$. Let $S=S_1 \cup S_2$ where $S_1=\cup^{t_1}_{i=1}V(P^{i})$ where $|V(P^{i})|\geq 2$ and $S_2=S-S_1$. Note that
\begin{itemize}
    \item all the edges of $K_{n}[S_2]$ receive the color $\{2,3\}$ or $\{1,3\}$. There is path $P'$ that contains all the vertices of $S_2$ and the color of the edges of $P'$ contains the element $3$. Let $v_p$ be one of end-vertices of $P'$

    \item the color of the edges in $E[v^{i}_{1},V(P^{j})]$ and $E[v^{i}_{2},V(P^{j})]$ is $\{2,3\}$ or $\{1,3\}$. This means that the color of each edge in $E[v^{i}_{1},V(P^{j})]\cup E[v^{i}_{2},V(P^{j})]$ contains the element $3$, where $1\leq i< j\leq t_1$. 
    
    \item the color of each edge in $E[v^{i}_{1},S_2]\cup E[v^{i}_{2},S_2]$ is $\{2,3\}$ or $\{1,3\}$. This means that the color of each edge in $E[v^{i}_{1},S_2]\cup E[v^{i}_{2},S_2]$ contains the  element $3$, where $1\leq i\leq t_1$.
\end{itemize}
For each $1\leq i\leq t_1$, let $P^{i}=v^{i}_{1}v^{i}_{3}v^{i}_{4}\ldots v^{i}_{\ell_i}v^{i}_{2}$ where $\ell_i\geq 2$. Without loss of generality, if $1\leq i\leq t_2$, then $\ell_i\geq 3$, if $t_2+1\leq i\leq t_1$, then $\ell_i=2$. For any $1\leq i \leq t_2$, in order to avoid a monochromatic $K_3$ with element $1$, for any $v^{i}_p,v^{i}_q\in V(P^{i})$, if $d_{P^i}(v^i_p,v^i_q)=2$, then $\chi(v^{i}_{p}v^{i}_{q})=\{2,3\}$. If $\ell_i$ is odd, then let $P^{i,1}=v^{i}_{1}v^{i}_{4}v^{i}_{6}\ldots v^{i}_{\ell_i-1}v^{i}_{2}$ and $P^{i,2}=v^{i}_{3}v^{i}_{5}v^{i}_{7}\ldots v^{i}_{\ell_i}$. Since $\ell_i\geq 3$, it means that $|V(P^{i,1})|\geq 2$ and $|V(P^{i,2})|\geq 1$. If $\ell_i$ is even, then let $P^{i,1}=v^{i}_{1}v^{i}_{4}v^{i}_{6}\ldots v^{i}_{\ell_i-2}v^{i}_{\ell_i}$ and $P^{i,2}=v^{i}_{3}v^{i}_{5}v^{i}_{7}\ldots v^{i}_{\ell_i-1}v^{i}_{2}$. Since $\ell_i\geq 4$, it means that $|V(P^{i,1})|\geq 2$ and $|V(P^{i,2})|\geq 2$. 
Note that for any $e\in E(P^{i,1})\cup E(P^{i,2})$, $\chi(e)=\{2,3\}$.  
For any $1\leq i\leq t_2$, if $\ell_i$ is even, then $v^{i}_2\in V(P^{i,2})$, let $P^{i,3}=P^{i,2}-v^{i}_2$, if $\ell_i$ is odd, then $v^{i}_2\in V(P^{i,1})$,  let $P^{i,3}=P^{i,1}-v^{i}_2$. Without loss of generality, let $\ell_i$ with $1\leq i\leq p$ be even and $\ell_i$ with $p+1\leq i\leq t_2$ be odd. Then we have a path $P''$ with $n-|S_2|$ vertices in $K_{n}[S_1\cup \{v_n\}]$ and the color of each edge of $P''$ contains the element $3$, where the path $P''$ is obtained from two paths
$$
P_{1}''=v^{t_1}_{1}v^{t_1-1}_{1}\ldots v^{t_2+1}_{1}v^{t_2}_{2}v^{t_1}_{2}v^{t_1-1}_{2}\ldots v^{t_2+1}_{2}
$$
and
$$
P_{2}''=P^{t_2,3}v^{t_{2}-1}_{2}P^{t_2,2}P^{(t_2-1),3}v^{t_{2}-2}P^{(t_2-1),2}\ldots P^{p+1,3}v^{p}P^{p+1,2}P^{p,1}v^{p-1}P^{p,3}P^{p-1,1}v^{p-2}P^{p,3}\ldots
P^{1,1}v_{n}P^{1,3}.
$$
by adding an edge between $v^{t_2+1}_{2}$ of $P_{1}''$ and the end-vertex $v^{t_2}_{1}$ of $P^{t_2,3}$ in $P_{2}''$. 
Then $P'\cup P''\cup \{v^{t_1}_{1}v_p\}$ is a monochromatic path $P_n$ with a color containing the element $3$, a contradiction. 
\end{proof}

%Recall that a path or cycle which contains every vertex of a graph is called a \textit{Hamilton path} or \textit{Hamilton cycle} of the graph. A graph is called a Hamiltonian graph if it contains a Hamiltonian cycle. We may ask how large the minimum degree must be in order to guarantee the existence of a Hamilton cycle. The following theorem of Dirac answers this question.
%\begin{theorem}{\rm \cite{D1952} }\label{T1}
%Let $G$ be a simple graph of minimum degree $\delta$, where $\delta\geq n/2$ and $n\geq 3$. Then $G$ is hamiltonian.
%\end{theorem}

Recall that, for a graph $H$ and a positive integer $N$, the Tu\'{r}an number $ex(N,H)$ is the maximum number of edges in an $H$-free graph on $N$ vertices.

\begin{theorem}\label{EX}
Let $r,s$ be two positive integers with $r>s\geq 2$ and let $G_1,\ldots,G_r$ be graphs. If $\sum^{r}_{i=1}ex(N,G_i)/{N\choose 2}<s$, then 
$$
\mathrm{R}_{r, s} (G_{1},\ldots,G_{r})\leq N.
$$    
\end{theorem}
\begin{proof}
Assume, to the contrary, that $\mathrm{R}_{r, s} (G_{1},\ldots,G_{r})>N$. Then there exists an $(r,s)$-coloring $\chi$ of $K_N$ without a  monochromatic copy of $G_i$ for each $i \ (1\leq i\leq r)$, which means that there exists an edge in $G_i$ such that its color does not contain the element $i$. 
Let $G$ be the $r$-colored $K_N$ under the coloring $\chi$. For each  $i \ (1\leq i\leq r)$, let $F_i$ be a spanning subgraph of $G$ induced by the edges of $G$, whose colors contain the element $i$. If there exists some $G_j$ such that $G_j$ is a subgraph of $F_j$, then there is a monochromatic copy of $G_j$ under the coloring $\chi$, a contradiction. 
If each $F_i$ does not contain a monochromatic copy of $G_i$, then each $F_i$ has at most $ex(N,G_i)$ edges where $1\leq i\leq r$. For each $e\in E[G]$, we have $|\chi(e)|=s$. Note that $\sum^{r}_{i=1}ex(N,G_i)\geq \sum_{i=1}^rE(F_i)={N\choose 2}s$. It implies that $\sum^{r}_{i=1}ex(N,G_i)/{N\choose 2}\geq s$, which contradicts to the fact that $\sum^{r}_{i=1}ex(N,G_i)/{N\choose 2}<s$. This means that $\mathrm{R}_{r, s} (G_{1},\ldots,G_{r})\leq N$.
\end{proof}

One of the oldest problems is the question of determining $ex(n,P_k)$
\begin{theorem}\upshape{\cite{Kopylov 1977,Faudree 1975}}\label{Kopylov 1977}
Let $n\equiv p~(\bmod~k-1)$, $0\leq p<k-1, k\geq 2$. Then  
$$ex(n,P_k)=\frac{1}{2}(k-2)n-\frac{1}{2}p(k-1-p).$$
\end{theorem}

The following corollary is immediate from Theorems \ref{EX} and \ref{Kopylov 1977}. 

\begin{corollary}\label{cor-PS}
Let $r,s$ be two positive integers with $r>s$ and let $P_{n_1},\ldots, P_{n_r}$ are $r$ paths with $n_i\geq 2$ vertices. Then 
$$
\mathrm{R}_{r,s} (P_{n_1},\ldots,P_{n_r})\leq \sum^{r}_{i=1}\frac{n_i-2}{2s}+\frac{1}{2}.
$$  
\end{corollary}
\begin{proof}
Let $N\equiv p_i~(\bmod~n_i-1)$, $0\leq p_i<n_i-1$. By Theorem \ref{Kopylov 1977}, we have  $ex(N,P_{n_i})=\frac{1}{2}(n_i-2)N-\frac{1}{2}p_i(n_i-1-p_i)$ where $1\leq i\leq r$. By Theorem \ref{EX}, if $\sum^{r}_{i=1}ex(N,P_{n_i})/{N\choose 2}<s$, then
$\mathrm{R}_{r,s} (P_{n_1},\ldots,P_{n_r})\leq N$. In order to get $\sum^{r}_{i=1}ex(N,P_{n_i})/{N\choose 2}<s$, we just need 
$$
sN^{2}-\left(\sum^{r}_{i=1}(n_i-2)+s\right)N+\sum^{r}_{i=1}p_i(n_i-1-p_i)>0,
$$ 
that is,
$$
N \leq \frac{\sum^{r}_{i=1}(n_i-2)+s-\sqrt{(\sum^{r}_{i=1}(n_i-2)+s)^{2}-4s\sum^{r}_{i=1}p_i(n_i-1-p_i)}}{2s}.
$$
Therefore, we have $\mathrm{R}_{r,s} (P_{n_1},\ldots,P_{n_r})\leq N\leq \sum^{r}_{i=1}\frac{n_i-2}{2s}+\frac{1}{2}$.
\end{proof}

\begin{theorem}\label{th-PL}
Let $r,s$ be two positive integers with $r>s$ and let $P_{n_1},\ldots, P_{n_r}$ be $r$ paths with $2\leq n_i\leq n_j$ and  $1\leq i\leq j\leq r$. Then 
$$
\mathrm{R}_{r, s} (P_{n_1},\ldots,P_{n_r})\geq \left\lfloor\frac{n_1-1}{2}\right\rfloor \left\lfloor\frac{r}{s}\right\rfloor+1.
$$
\end{theorem}
\begin{proof}
Let $N=\lfloor\frac{n_1-1}{2}\rfloor \lfloor r/s \rfloor$. We construct a $(r,s)$-coloring $\chi$ of the edges of $K_{N}$ such that there is no  monochromatic $P_{n_i}$. This means that there exists an edge such that its color does not contain the element $i$. Let $\{1, \ldots,r\}$ be the set of color elements, and let $\mathcal{C}=\{C_i \,|\, 1\leq i\leq \lfloor r/s\rfloor \}$, where $C_i=\{{(i-1)s+1},\ldots,{is}\}$ be the set of $s$-color elements. 
\begin{fact}\label{fact1}
For each $C_i,C_j \ (1\leq i\neq j \leq \lfloor r/s\rfloor)$, we have $C_i\cap C_j=\emptyset$.     
\end{fact}

Partition the vertices of $V(K_{N})$ such that $V(K_{N})=V_1\cup \ldots \cup V_{\lfloor r/s\rfloor}$ with $|V_{i}|=\lfloor (n_1-1)/{2}\rfloor$ for $1\leq i\leq \lfloor r/s\rfloor$. We contract each clique $V_i$ to be a vertex $v_i$ for each $i \ (1\leq i\leq \lfloor r/s\rfloor)$, and the resulting graph is a clique $K_{\lfloor r/s\rfloor}$ with vertex set $V(K_{\lfloor r/s\rfloor)})=\{v_i\,|\,1\leq i\leq \lfloor r/s\rfloor\}$. We now give the coloring $\chi$ of $K_{N}$ by giving a coloring $\psi$ of the edges of $K_{\lfloor r/s\rfloor}$ such that
\begin{equation}\label{Eq-coloring}
\psi(v_iv_j)=\chi(e)
\end{equation}
for $v_i,v_j\in V(K_{\lfloor r/s\rfloor})$ and $e\in E_{K_n}[V_i,V_j]$.
The coloring $\chi$ can be given as follows. 
\begin{itemize}
    \item If $\lfloor r/s\rfloor$ is even, then it follows from Theorem \ref{KM} that the complete graph $K_{\lfloor r/s\rfloor}$ can be decomposed into $(\lfloor r/s\rfloor-1)$ edge-disjoint perfect matchings, say $\mathcal{M}_1,\ldots, \mathcal{M}_{\lfloor r/s\rfloor-1}$. Let $\psi(e)=C_i$ if $e\in \mathcal{M}_i$ where $1\leq i\leq \lfloor r/s\rfloor-1$. This brings the coloring $\chi$ satisfying Eq. (\ref{Eq-coloring}) and
$\chi(e)=C_{\lfloor r/s\rfloor}$ for $e\in E(K_{N}[V_i])$ where $1\leq i\leq \lfloor r/s\rfloor$.

\item If $\lfloor r/s\rfloor$ is odd, then it follows from Theorem \ref{KM-odd} that the complete graph $K_{\lfloor r/s\rfloor}$ can be decomposed into edge-disjoint
$\lfloor r/s \rfloor$ maximal matchings $\mathcal{M}_1,\ldots,\mathcal{M}_{\lfloor r/s\rfloor}$ and each $\mathcal{M}_i(1\leq i\leq \lfloor r/s\rfloor)$ not cover only one vertex of $V(K_{2n+1})$. Let $\psi(e)=C_i$ if $e\in \mathcal{M}_i$ where $1\leq i\leq \lfloor r/s\rfloor$. Since for any vertex $v_i\in V(K_{\lfloor r/s\rfloor})$, $d_{K_{\lfloor r/s\rfloor}}(v_i)=\lfloor r/s \rfloor-1$, and  hence there exists some $C_j\in \mathcal{C}$ such that  $C_j\notin \{\psi(v_iv_j) \,|\, v_j\in N_{K_{r/s}}(v_i)\}$. Let $\chi(e)=C_{j}$ for $e\in E(K_{N}[V_i])$ where $1\leq i\leq \lfloor r/s\rfloor$. 
\end{itemize}

Under the coloring $\chi$, it suffices to show that there is no monochromatic copy of $P_{n_i}$. 

Suppose that $\lfloor r/s \rfloor$ is even. Note that the color $C_{\lfloor r/s\rfloor}$ only appear on the edges of the cliques $K_{N}[V_j]$, where $1\leq j\leq \lfloor r/s\rfloor$. 
For any element $j\in[1,\lfloor r/s \rfloor-s]$, since  $C_i=\{{(i-1)s+1},\ldots,{is}\}$ where $1\leq i\leq \lfloor r/s\rfloor$, it follows that $j\in C_p$ for some $p \ (1\leq p\leq \lfloor r/s\rfloor)$. From Fact \ref{fact1}, $C_p$ is the unique color containing the element $j$. 
Since $\psi(e)=C_p$ for each $e\in \mathcal{M}_p$, it follows that the subgraph graph in $K_{\lfloor r/s\rfloor}$ induced by the edges whose color contains the element $j$ form a perfect matching $\mathcal{M}_p$. That is, $\mathcal{M}_p$ is colored by $C_p$. 
Let $E(\mathcal{M}_p)=\{v_{i_{2j-1}}v_{i_{2j}}\,|\,1\leq j\leq \frac{1}{2}\lfloor r/s\rfloor\}$. Then the subgraph induced by the edges in 
$$
\left\{e\in E_{K_N}[V_{i_{2j-1}},V_{i_{2j}}]\,\Big|\, 1\leq j\leq \frac{1}{2}\lfloor r/s\rfloor\right\}
$$
is the union of vertex disjoint complete bipartite graphs $K_{|V_{i_1}|,|V_{i_2}|},\ldots,K_{|V_{i_{\lfloor r/s\rfloor}-1}|,|V_{i_{\lfloor r/s\rfloor}}|}$. 
Then the length of the longest path whose color contains the element $j \ (1 \leq j\leq \lfloor r/s \rfloor-s)$ is at most 
$2\lfloor\frac{n_1-1}{2}\rfloor$, since $N=\lfloor\frac{n_1-1}{2}\rfloor \lfloor r/s \rfloor$. 
For any element $\lfloor r/s \rfloor-s+1 \leq j\leq \lfloor r/s \rfloor$, since $j\in C_{\lfloor r/s \rfloor}$ and $\chi(e)=C_{\lfloor r/s\rfloor}$ for each $e\in E(K_{N}[V_i])$ where $1\leq i\leq \lfloor r/s\rfloor$, it follows that the length of the longest path containing the element $j \ (\lfloor r/s \rfloor-s+1 \leq j\leq \lfloor r/s \rfloor)$ is at most $\lfloor\frac{n_1-1}{2}\rfloor$. Clearly, there is no monochromatic copy of $P_{n_i}$ in $K_{N}$ under the coloring $\chi$. 

Suppose that $\lfloor r/s \rfloor$ is odd. For any element $j\in[1,\lfloor r/s \rfloor]$, since  $C_i=\{{(i-1)s+1},\ldots,{is}\}$ where $1\leq i\leq \lfloor r/s\rfloor$, it follows that $j\in C_q$ for some $q \ (1\leq p\leq \lfloor r/s\rfloor)$. From Fact \ref{fact1}, $C_q$ is the unique color containing the element $j$. Since $\psi(e)=C_q$ if $e\in \mathcal{M}_q$ where $\mathcal{M}_q$ is a maximal matching and there is only one vertex of $V(K_{\lfloor r/s\rfloor})$ not covered by $\mathcal{M}_q$, it follows that the subgraph graph in $K_{\lfloor r/s\rfloor}$ induced by the edges whose color contains the element $j$ form a maximal matching $\mathcal{M}_q$ where there is only one vertex of $V(K_{\lfloor r/s\rfloor})$ not covered. That is, $\mathcal{M}_q$ is colored by $C_p$. Let $E(\mathcal{M}_q)=\{v_{i_{2j-1}}v_{i_{2j}}\,|\,1\leq j\leq \frac{1}{2}(\lfloor r/s\rfloor-1)\}$ and $v_{i_{\lfloor r/s\rfloor}}$ is the vertex not covered by $\mathcal{M}_q$.
Then the subgraph induced by the edges in 
$$
\left\{e\in E_{K_N}[V_{i_{2j-1}},V_{i_{2j}}]\,\Big|\, 1\leq j\leq \frac{1}{2}(\lfloor r/s\rfloor-1)\right\}
$$
is the union of vertex disjoint complete bipartite graphs $K_{|V_{i_1}|,|V_{i_2}|},\ldots,K_{|V_{i_{\lfloor r/s\rfloor-2}}|,|V_{i_{\lfloor r/s\rfloor-1}}|}$. Note that $C_q\notin \{\psi(v_{i_{\lfloor r/s\rfloor}}v_{i_{j}}) \,|\, v_{i_{j}}\in N_{K_{r/s}}(v_{i_{\lfloor r/s\rfloor}})\}$. Otherwise there is a vertex $v_{i_{j}}\neq v_{i_{\lfloor r/s\rfloor}}\in V(K_{\lfloor r/s\rfloor})$ such that
$\psi(v_{i_{\lfloor r/s\rfloor}}v_{i_{j}})=C_q$. Since $v_{i_{j}}$ is covered by $\mathcal{M}_q$, it implies that there exists $v_{i_{t}}\in V(\mathcal{M}_q)$ such that $v_{i_{t}}v_{i_{j}}\in \mathcal{M}_q$. Then $\psi(v_{i_{\lfloor r/s\rfloor}}v_{i_{j}})=\psi(v_{i_{t}}v_{i_{j}}) =C_q$, it contradicted to $\psi(e)=C_q$ if $e\in \mathcal{M}_q$. Hence $C_q\notin \{\psi(v_{i_{\lfloor r/s\rfloor}}v_{i_{j}}) \,|\, v_{i_{j}}\in N_{K_{r/s}}(v_{i_{\lfloor r/s\rfloor}})\}$. Since $\chi(e)=C_{q}$ for $e\in E(K_{N}[V_{i_{\lfloor r/s\rfloor}}])$, it follows that the induced graph in $K_{N}$ which the edges contain element $j$ are $K_{|V_{i_1}|,|V_{i_2}|}\cup\ldots \cup K_{|V_{i_{\lfloor r/s\rfloor}}|-2,|V_{i_{\lfloor r/s\rfloor}}|-1}\cup K_{N}[V_{i_{\lfloor r/s\rfloor}}]$  where $V(K_{N}[V_{i_{\lfloor r/s\rfloor}}])\cap V(K_{|V_{i_{2j-1}}|,|V_{i_{2j}}|})=\emptyset$ and  $1\leq j\leq \frac{1}{2}(\lfloor r/s\rfloor-1)$. Hence the length of the longest path whose color contains the element $j$ is at most $2\lfloor\frac{n_1-1}{2}\rfloor$ where $1\leq j \leq \lfloor r/s \rfloor$ and there is no monochromatic copy of $P_{n_i}$ in $K_{N}$ under the coloring $\chi$. It implies that $\mathrm{R}_{r, s} (P_{n_1},\ldots,P_{n_r})\geq N$.
\end{proof}

The following corollary is immediate. 
\begin{corollary}\label{cor:A-Pn-Pn}
If $s|r$ and $n\geq 2$ is even, then
$$
\mathrm{R}_{r, s} (P_{n})=\frac{r(n-2)}{2s}.
$$
\end{corollary}
\begin{proof}
From Corollary \ref{cor-PS}, we have $\mathrm{R}_{r, s} (P_{n})\leq \frac{r(n-2)}{2s}+\frac{1}{2}$. From Theorem \ref{th-PL}, since $s|r$ and $n\geq 2$ is even, we have $\mathrm{R}_{r, s} (P_{n})\geq \lfloor\frac{n-1}{2}\rfloor \frac{r}{s}=\frac{r(n-2)}{2s}$. This implies that $\frac{r(n-2)}{2s}\leq  \mathrm{R}_{r, s} (P_{n})\leq \frac{r(n-2)}{2s}+\frac{1}{2}$, and hence $\mathrm{R}_{r, s} (P_{n})=\frac{r(n-2)}{2s}$.
\end{proof}

\subsection{Stars}

In this section, we consider the $(r,s)$-coloring for star. Let $K_{1,n}$ be the star with $n+1$ vertices where one central vertex and $n$ leaves.  

\begin{theorem}\label{S1}
Let $n_1,\ldots, n_r,r,s$ be positive integers with $3\leq n_1\leq \cdots\leq n_r$ and $r>s$. If $\lfloor \frac{rn_1-r}{s}\rfloor\leq {r\choose s}$, then
$$
\left\lfloor \frac{rn_1-r}{s}\right\rfloor+1\leq  \mathrm{R}_{r, s} (K_{1,n_1},\ldots,K_{1,n_r})\leq \left\lceil \frac{\sum_{i=1}^rn_i-r+1}{s}\right\rceil+1.
$$
Moreover, the upper bound is sharp. 
\end{theorem}
\begin{proof}
For the upper bound, let $N=\left\lceil (\sum_{i=1}^rn_i-r+1)/s \right\rceil+1$ and let $\{1,\ldots,r\}$ be the color set. For any $(r,s)$-coloring $\chi$ of the edges of $K_N$, suppose that there is no monochromatic copy of $K_{1,n_i}$ with a color containing $i$ as its element for each $i \ (1\leq i\leq r)$. Without loss of generality, we assume that $K_{1,m_i}$ is a star with maximal edges such that its color contains the element $i$ under $\chi$ for each $i \ (1\leq i\leq r)$. To avoid a monochromatic copy of $K_{1,n_i}$, we have $m_i\leq n_i-1$ for any $i \ (1\leq i\leq r)$. 
For any vertex $v\in V(K_N)$, we have $d_{K_N}(v)=\left\lceil (\sum_{i=1}^rn_i-r+1)/s \right\rceil$, and hence the number of edges incident to $v$ with colors in ${r\choose s}$ is 
$$
sd_{K_N}(v)=s\left\lceil \frac{\sum_{i=1}^rn_i-r+1}{s}\right\rceil \geq s\left(
 \frac{\sum_{i=1}^rn_i-r+1}{s}\right)=\sum^{r}_{i=1}n_i-r+1>\sum^{r}_{i=1}m_i,
$$
a contradiction. Then there exists some monochromatic copy of $K_{1,n_i}$ with a color containing the element $i$ where $1\leq i\leq r$ under the coloring $\chi$, and therefore $\mathrm{R}_{r, s} (K_{1,n_1},\ldots,K_{1,n_r})\leq N$. 

For the lower bound, let $M=\left\lfloor (rn_1-r)/s\right\rfloor$ and let $R=\{1,\ldots,r\}$ be the color set. We construct a $(r,s)$-coloring $\chi$ of $K_{M}$ such that for each  $i \ (1\leq i\leq r)$ there is no monochromatic copy of $K_{1,n_i}$ with a color containing element $i$ under the coloring $\chi$.
Since $M=\left\lfloor (rn_1-r)/s \right\rfloor$ and $M\leq {r\choose s}$, let $sM=rn_1-r-p$ where $0\leq p\leq s-1$. 

\begin{fact}
There exists a set $\mathcal{C}=\{C_1,\ldots, C_M\}$ of $s$-subsets of $R$ with $M$ elements such that for any $i\in R$, 
\begin{itemize}
    \item if $1\leq i\leq r-p$, then there are just $n_1-1$ $s$-set $C_{j_1},\ldots, C_{j_{n_1-1}}$ in $\mathcal{C}$ such that $i\in C_{j_{\ell}}, 1\leq \ell\leq n_1-1$. 

    \item if $r-p+1\leq i\leq r$, then there are just $n_1-2$ $s$-set $C_{q_1},\ldots, C_{q_{n_1-2}}$ in $\mathcal{C}$ such that $i\in C_{q_{\ell}}, 1\leq \ell\leq n_1-2$.
\end{itemize}
\end{fact}

\begin{proof}
Construct a $r\times (n_1-1)$ matrix $A$ such that for each row $i$, if $1\leq i\leq r-p$, then the $n_1-1$ elements are all $i$; if $r-p+1\leq i\leq r$, then the first $n_1-2$ elements are all $i$; and the rest elements in this matrix are $0$. Give a corresponding relation $\psi:A\longrightarrow B$ by the position of the elements in terms of the columns. 
\begin{equation*}
\left.
\begin{bmatrix}
1 & 1 &  \ldots & 1 & 1 \\
\vdots  & \vdots  &  & \vdots  & \vdots  \\
r-p & r-p &  \ldots & r-p & r-p \\

r-p+1 & r-p+1 & \ldots & r-p+1 & 0 \\
\vdots  & \vdots  &  & \vdots  & \vdots\\
r & r  & \ldots & r & 0 \\
\end{bmatrix}
\right.
\rightarrow
\left. 
\begin{bmatrix}
a_1 & a_{r+1} &  \ldots & a_{(n_1-3)r+1} & a_{(n_1-2)r+1}\\
\vdots  & \vdots  &  & \vdots  & \vdots \\
a_{r-p} & a_{2r-p} &  \ldots & a_{(n_1-2)r-p} & a_{(n_1-1)r-p} \\

a_{r-p+1} & a_{2r-p+1} & \ldots & a_{(n_1-2)r-p+1} & 0  \\
\vdots  & \vdots  &  & \vdots  & \vdots \\
a_r & a_{2r} & \ldots & a_{(n_1-2)r}&  0 \\
\end{bmatrix}
\right.
\end{equation*}
Let $C_i=\{\psi^{-1}(a_{(i-1)s+1}),\ldots,\psi^{-1}(a_{is})\}$ where $1\leq i\leq M$. Then $\mathcal{C}=\{C_1,\ldots, C_M\}$ is the required set. 
\end{proof}

We now give the edge-coloring $\chi$ of $K_N$ in terms of the set $\mathcal{C}$. 

If $M$ is even, then it follows from Theorem \ref{KM} that the complete graph $K_{M}$ can be decomposed into $(M-1)$ edge-disjoint perfect matchings, say $\mathcal{M}_1,\ldots, \mathcal{M}_{M-1}$. Let $\chi(e)=C_i$ for any $e\in \mathcal{M}_i$, $C_i\in \mathcal{C}$ and $1\leq i\leq M-1$. If $M$ is odd, then it follows from Theorem \ref{KM-odd} that the complete graph $K_{M}$ can be decomposed into $M$ edge-disjoint maximal matchings, say $\mathcal{M}_1,\ldots, \mathcal{M}_{M}$, such that each matching does not cover only one vertex of $V(K_{M})$. Let $\chi(e)=C_i$ for each $e\in \mathcal{M}_i$, $C_i\in \mathcal{C}$ and $1\leq i\leq M$.

For any vertex $v\in K_M$, if $M$ is even, then $\bigcup^{M-1}_{i=1}\mathcal{M}_i$ is an edge decomposition of $E[K_M]$; if $M$ is odd, $\bigcup^{M}_{i=1}\mathcal{M}_i$ is an edge decomposition of $E[K_M]$. This implies that there is a $p \ (1\leq p\leq M)$ such that $vv_j\in \mathcal{M}_p$ where $v_j\in N(v)$. Since each $\mathcal{M}_i$ is a matching, for any vertex $v_w\neq v_j \in N(v)$, $vv_w\in \mathcal{M}_q$ and $p\neq q$. Hence $\chi(vv_j)=C_p$, $\chi(vv_w)=C_q$ and $p\neq q$. For any $vv_j\in E(K_M)$, $\chi(vv_j)\in \mathcal{C}$, it follows by the construction of $\mathcal{C}$ that each element $i$ appear at most $n_1-1$ times on the edges incident with the vertex $v$. Hence there is no monochromatic $K_{1,n_i}$ with element $i$.
%Let $C=\{c_1,\ldots, c_{M}\}$ where $c_i\subseteq {r\choose s},1\leq i\leq M$ be  defined as follows. 
%\begin{itemize}
   % \item If $p=0$, then $sM=\sum^{r}_{i=1}n_i-r$.    Let $c_{i,1},c_{i,2},...,c_{i,n_i-1}$ be the colors containing $i$, where $1\leq i\leq r$. Then $C=\{c_{i,j}\,|\,1\leq i\leq r, ~1\leq j\leq n_i-1\}$. 

    %\item If $p\not=0$,  then let $c_{i,1},c_{i,2},...,c_{i,n_i-1}$ be the colors in $C$ containing $i$, where $1\leq i\leq r-p$. Then $C'=\{c_{i,j}\,|\,1\leq i\leq r-p, ~1\leq j\leq n_i-1\}$. 
   % Let $c_{i,1},c_{i,2},...,c_{i,n_i-2}$ be the colors in $C$ containing $i$, where $r-p+1\leq i\leq r$. Then 
   % $C''=\{c_{i,j}\,|\,r-p+1\leq i\leq r, ~1\leq j\leq n_i-2\}$.
%Then $C=C'\cup C''$. 
%\end{itemize}

%Let $S_0=\emptyset$.
%Firstly, we select a vertex $v_1\in V(K_M)$, and any color $c_{i,j}\in C$. We color each edge incident to $v_1$ with a color in $C-c_{i,j}$ such that any two edges with different colors in $C-c_{i,j}$. Let 
%$S_1=S_0\cup v_1$. Select a vertex $v_2\in V(K_N)-S_1$

%Then select a vertex $w\in V(K_N)-S$ and a $c_i\in C$ such that $C-c_i$ is a coloring way of the edges incident to $w$ and $S=S\cup w$. Repeated this process until $S=V(K_N)$. Note that  the $c_i$ may be same. According this coloring way, $E(K_N)$ are colored and for any $v\in V(K_N)$, there are at most $K_{1,n_i-1}$ where $1\leq i\leq r$. Hence there is no monochromatic $K_{1,n_i}$ with color $i$.
\end{proof}
Next we will state that the upper bound of Theorem \ref{S1} is sharp.

\begin{proposition}\label{S2}
For $n\geq 2$, we have 
$$
\mathrm{R}_{3, 2} (K_{1,2},K_{1,n},K_{1,n})=n+1.
$$     
\end{proposition}

\begin{proof}
From Theorem \ref{S1}, we have $\mathrm{R}_{3, 2} (K_{1,2},K_{1,n},K_{1,n})\leq n+1.$ To show the lower bound, we let $\{1,2,3\}$ be the color element set. We give a $(3,2)$-coloring $\chi$ of $K_n$ by coloring each edge of $K_n$ with $\{2,3\}$. Since there is no $K_{1,n}$ in $K_n$, it follows that there is no monochromatic copy of $K_{1,n}$ with the color containing the element $2$ or $3$. Clearly, there is no monochromatic copy of $K_{1.2}$ with the color containing $1$ under the coloring $\chi$. This implies that $\mathrm{R}_{3, 2} (K_{1,2},K_{1,n},K_{1,n})\geq n+1$, and hence $\mathrm{R}_{3, 2} (K_{1,2},K_{1,n},K_{1,n})= n+1.$
\end{proof} 

\section{General lower bound for $(r,s)$-coloring}

We extend the result $\mathrm{R}_{r,s}(G)$ in \cite{Le22} where $G$ is a completed graph to the case where  $G$ is a general graph.

\begin{theorem}\label{T1}
For two positive integers $r,s$ with $r\geq s$, we have   

$(i)$ $\mathrm{R}_{r,s}(G)\leq \mathrm{R}_{r+1,s}(G)$;

$(ii)$ If $s\geq 2$ then $\mathrm{R}_{r,s}(G)\leq \mathrm{R}_{r-1,s-1}(G)$;

$(iii)$ If $s\geq 2$ then  $\mathrm{R}_{r,s}(G)\leq \mathrm{R}_{r,s-1}(G)$.
\end{theorem}

\begin{proof}
For $(i)$, let $n=\mathrm{R}_{r,s}(G)-1$. Then 
there exists a $(r,s)$-coloring $\chi$ of $K_n$ such that there is no monochromatic copy of $G$ under $\chi$. Since a $(r,s)$-coloring of $K_n$ is also a $(r+1,s)$-coloring of $K_n$, it follows that $\chi$ is a $(r+1)$-coloring of $K_n$ such that there is no monochromatic copy of $G$. This means that $\mathrm{R}_{r+1,s}(G)\geq n+1=\mathrm{R}_{r,s}(G)$.

For $(ii)$, let $n=\mathrm{R}_{r,s}(G)-1$. Then there exists a $(r,s)$-coloring $\chi$ of $K_n$ such that there is no monochromatic copy of $G$ under $\chi$. Define a new coloring $\chi'$ of $K_n$ from $\chi$ by removing one color element on each edge, that is, for any $uv\in E(K_n)$, let $\chi'(uv)=\chi(uv)-\max \{c(uv)\}$. Note that for any edge $uv\in E(K_n)$, then $r \notin \chi'(uv)$ and $|\chi'(uv)|=s-1$. Hence $\chi'$ is an $(r-1,s-1)$-coloring
of $K_n$ and there is no monochromatic $G$ under $\chi'$. Thus $\mathrm{R}_{r-1,s-1}(G)\geq n+1\geq \mathrm{R}_{r,s}(G)$. 

For the $(iii)$, since $\mathrm{R}_{r,s}(G)\leq \mathrm{R}_{r+1,s}(G)$ by $(i)$, we have $\mathrm{R}_{r,s}(G)\leq \mathrm{R}_{r+1,s}(G)\leq \mathrm{R}_{r,s-1}(G)$, where the second inequality is gotten by $(ii)$.
\end{proof}

The following corollary is immediate by Theorem \ref{T1} $(iii)$.
\begin{corollary}
For all positive integers $r,s,r\geq s$, we have 
$$
\mathrm{R}_{r,s}(G)\leq \mathrm{R}_{r,1}(G)=\mathrm{R}_r(G).$$   
\end{corollary}

Le \cite{Le22} gave the lower bound of $K_{n}$ that is  $\mathrm{R}_{r,s}(K_{n})\geq \lfloor \frac{n}{e} r^{-1/n}(\frac{r}{s})^{(n-1)/2}\rfloor$ where $r\geq s$ and $n\geq 3$. In this section, we gave the lower bound of $\mathrm{R}_{r,s}(G_1,\ldots,G_r)$ where $G_i$ be a graph with $n_i$ vertices and $m_i$ edges and $1\leq i\leq r$.

Let $\Omega$ be a probability space and $A_{1},A_{2},\ldots,A_{n}$ be events. Let $G$ be a graph with vertex set $\{1,\ldots,n\}$. $G$ is a dependence graph of $\{A_{1},A_{2},\ldots,A_{n}\}$ if for $1\leq i\leq n$, $A_{i}$ is mutually independent of $\{A_{j}:\{i,j\}\notin G\}$.
\begin{theorem}{\upshape (Lov\'{a}sz Local Lemma \cite{Lovasz})}
Let $A_{1},\ldots, A_{n}$ be events in a probability space $\Omega$ with dependence graph $\Gamma$. Suppose there exist $x_{1},x_{2},\ldots,x_{n}$ such that $0<x_{i}\leq 1$ and
$$
\Pr[A_{i}]\leq(1-x_{i})\prod_{\{i,j\}\in\Gamma}x_{j}, 1\leq i\leq n.
$$
Then $\Pr[\underset{i}{\bigwedge}\overline{A_{i}}]>0$.
\end{theorem}
Set
$$
y_{i}=\frac{1-x_{i}}{x_{i}\Pr[A_{i}]}
$$
such that
$$
x_{i}=\frac{1}{1+y_{i}\Pr[A_{i}]}.
$$
Since $1+z\leq exp(z)$, we have
\begin{corollary}
Suppose that $A_{1},\ldots, A_{n}$ be events in a probability space $\Omega$ with dependence graph $\Gamma$. There exist positive  $y_{1},y_{2},\ldots,y_{n}$ satisfying
$$
\log y_{i}> \sum_{\{i,j\}\in\Gamma}y_{j}\Pr[A_{j}]+y_{i}\Pr[A_{i}]~~1\leq i\leq n.
$$
Then $\Pr[\underset{i}{\bigwedge}\overline{A_{i}}]>0$.
\end{corollary}

By Lov\'{a}sz Local Lemma, we gave the lower bounds for $\mathrm{R}_{r,s}(G)$.

\begin{theorem}\label{L1}
Let $G$ be a connected graph with $n\geq 4$ vertices and $m$ edges where ${n-1\choose 2}\leq m \leq {n\choose 2}$. For any positive integers $r,s$ with $r>s$, we have  
$$
\mathrm{R}_{r,s}(G)\geq \left(\frac{m-1}{c_{2}sr\ln ((m-1)/c_2) }\right)^{sr},
$$
where $c_1,c_2,c_3$ are three positive numbers with $c_1\leq \frac{r-s}{r}$, and $c_3-c_{1}c_{2}+3n<0$. 
\end{theorem}

\begin{proof}
Let 
$N=\mathrm{R}_{r,s}(G)$ and the edge of $K_{N}$ is independently $(r,s)$-coloring with the probability that each edge is colored with an $s$-subset of $[r]$ that contains the color element $r$ is $1-p$ and does not contain the color element $r$ being $p$. For any a edge $e\in E[K_{N}]$, let $B$ be the event that the color of $e$ contains the color element $i$ with $i\neq r$. Let $B_1$ be the event that the color of $e$ contains the color element $r$ and $B_2$ be the event that the color of $e$ does not contains the color element $r$. Then we have 
$$
\Pr[B]=\Pr[B|B_1]+\Pr[B|B_2]=(1-p)\frac{{r-2\choose s-2}}{{r-1\choose s-1}}+p\frac{{r-2\choose s-1}}{{r-1\choose s}}=\frac{s-1+p}{r-1}.
$$

For a vertex set $S\subset V(K_N)$ and $|S|=n$, let $A_{S}$ be the event that the induced graph by $S$ contains a monochromatic $G$. Then 
\begin{equation*}
\begin{split}
\Pr[A_{S}]&\leq 
\left((1-p)^{m}+(r-1)\left(\frac{s-1+p}{r-1}\right)^{m}\right) {{n\choose 2}\choose m}   \\[0.2cm]
&= \left((1-p)^{m}+(s+p-1)\left(\frac{s-1+p}{r-1}\right)^{m-1}\right) {{n\choose 2}\choose m}  \\[0.2cm]
&\leq \left((1-p)^{m}+(s+p-1)(1-p)^{m-1} \right) {{n\choose 2}\choose m}  \\[0.2cm]
&=s(1-p)^{m-1}{{n\choose 2}\choose m},
\end{split}
\end{equation*}  
where we can restrict $p$ such that $\frac{s-1+p}{r-1}\leq 1-p$.
If
$\Pr[\underset{S}{\bigwedge }\overline{A_{S}}]>0$,
then $\mathrm{R}_{r,s}(G)>N$.
Let $\Gamma$ be the graph with at most ${N\choose n}$ vertices corresponding to all $A_{S}$, where $\{A_{S},A_{S'}\}$ is an edge of $\Gamma$ if and only if two subgraphs $G'$ ($\cong G$) and $G''$ ($\cong G$) share at least one edge, that is, the events $A_{S}$ and $A_{S'}$ is independent. Let $N_{AA}$ denote the number of vertices of the form $A_{S}$ joined to some other vertex of this form.
If there exist two positive numbers $p,y$ such that
$$
\ln y>y\Pr[A_{S}](N_{AA}+1),
$$
then $\mathrm{R}_{r,s}(G)>N$. 
Note that
$$
N_{AA}+1\leq {N\choose n}\leq \left(\frac{Ne}{n}\right)^{n}.
$$
Set
$$
p=c_{1}N^{-\frac{1}{sr}}, \ \ 
m-1=c_2N^{\frac{1}{sr}}\ln N, \ \
y=e^{c_3\ln N}.
$$
Since ${a\choose b}\leq a^{b}$ and  ${n-1\choose 2}\leq m \leq {n\choose 2}$, it follows that 
\begin{equation*}
\begin{split}
y\Pr[A_{S}](N_{AA}+1)&\leq
yse^{-p(m-1)}\left(\frac{Ne}{n}\right)^{n}{{n\choose 2}\choose m}\\[0.2cm] 
&\leq yse^{-p(m-1)} 
\left(\frac{Ne}{n}\right)^{n}   {{n\choose 2}\choose n-1} (since ~{n-1\choose 2}\leq m \leq {n\choose 2})  \\[0.2cm]
&\leq  yse^{-p(m-1)}\left(\frac{Ne}{n}\right)^{n} n^{2n}\\[0.2cm]
&=yse^{-p(m-1)}(Ne)^{n}n^{2}\\[0.2cm]
&\leq  yse^{-p(m-1)}(Nen)^{n} \\[0.2cm]
&=se^{(c_3-c_1c_2)\ln N}e^{n\ln(Nen)} \\[0.2cm]
&=se^{(c_3-c_1c_2+n\ln(Nen)/\ln N)\ln N} \\[0.2cm]
&\leq se^{(c_3-c_1c_2+3n)\ln N}.
\end{split}
\end{equation*}  
If $c_3-c_1c_2+3n<0$, then there is a positive integer $N_1$ such that for the sufficient small $\epsilon$ and when $N>N_1$, we have
$$
y\Pr [A_{S}](N_{AA}+1)<\epsilon<c_3\ln N=\ln y.
$$ 
In order to get $\frac{s-1+p}{r-1}\leq 1-p$, we have $p=c_{1}N^{-\frac{1}{sr}}\leq 1-\frac{s}{r}$, then $c_1\leq 1-\frac{s}{r}$.
Since 
$$m-1=c_2N^{\frac{1}{sr}}\ln N\leq c_2N^{\frac{1}{sr}}rs\ln ((m-1)/c_2),$$
it follows that
$$
\mathrm{R}_{r,s}(G)=N \geq \left(\frac{m-1}{c_{2}sr\ln ((m-1)/c_2) }\right)^{sr}.
$$
\end{proof}

\begin{theorem}\label{L2}
Let $r,s$ be two positive integers   with $r>s$, and let $G_i$ be a graph with $n_i$ vertices and $m_i$ edges. Then 
$$
\mathrm{R}_{r,s}(G_1,\ldots,G_r)\geq \frac{n^{n/n'}}{e}\left(\left(\frac{r}{s}\right)^{m}\frac{(2m)^{m}}{\sum^{r}_{i=1}(en^{2}_i)^{m_i}}\right)^{1/n'},
$$
where $m=\min\{m_1,\ldots,m_r\}$, $n=\min\{n_1,\ldots,n_r\}$, and $n'=\max\{n_1,\ldots,n_r\}$.
\end{theorem}
\begin{proof}
Let 
\begin{equation}\label{eq2-1}
N=\frac{n^{n/n'}}{e}\left(\left(\frac{r}{s}\right)^{m}\frac{(2m)^{m})}{\sum^{r}_{i=1}(en^{2}_i)^{m_i}}\right)^{1/n'}-1.
\end{equation}
For the complete graph $K_N$, we color the edges of $K_N$ uniformly and independently at random, such that each edge receives a color, which is an $s$-subset of $[r]$ with probability $1/{r\choose s}$. We consider any vertex set $U_i\subseteq V(K_{N})$ with $|U_i|=n_i$ for each $i \ (1\leq i\leq r)$. For any $e\in E(K_N[U_i])$, if the color of $e$ contains $i$, then $e$ has the other $(s-1)$ color elements which can be chosen
from $[r]-\{i\}$. Therefore, for any edge $e\in E(K_N[U_i])$, we have
$$
\Pr[e~\text{contains~color}~i]=\frac{{r-1\choose s-1}}{{r\choose s}}=\frac{s}{r}.
$$
Thus the probability of $K_N[U_i]$ contains at least a monochromatic copy of $G_i$ whose each edge receives the color containing $i$ is no more than 
$$
\left(\frac{s}{r}\right)^{m_i}{{n_i\choose 2}\choose m_i}\leq \left(\frac{n^{2}_ie}{2m_i}\right)^{m_i}\left(\frac{s}{r}\right)^{m_i}.
$$
Let $A_i$ be the even that there is a monochromatic copy of $G_i$ such that all the edges receive the color containing $i$ in $K_N$ for each $i \ (1\leq i\leq r)$. From the arbitrariness of $U_i$, we have 
\begin{align*}
\Pr[A_i]&\leq \Pr[\cup_{U_i\in V(K_N), |U_i|=n_i} K_N[U_i]~\text{contains~at~ least~a~monochromatic}~G_i]\\[0.2cm]
&\leq \sum_{U_i\in V(K_N), |U_i|=n_i}\Pr[K_N[U_i]~\text{contains~at~ least~a~monochromatic}~G_i]\\[0.2cm]
&={N\choose n_i}\left(\frac{s}{r}\right)^{m_i}{{n_i\choose 2}\choose m_i}\\[0.2cm]
&\leq \left(\frac{Ne}{n_i}\right)^{n_i}\left(\frac{n^{2}_ie}{2m_i}\right)^{m_i}\left(\frac{s}{r}\right)^{m_i}.
\end{align*}
Then
\begin{align*}
\Pr\left[\bigcup^{r}_{i=1}A_i\right]&\leq \sum^{r}_{i=1}\Pr[A_i]\leq \sum^{r}_{i=1}\left(\frac{Ne}{n_i}\right)^{n_i}\left(\frac{n^{2}_ie}{2m_i}\right)^{m_i}\left(\frac{s}{r}\right)^{m_i}\\[0.2cm]
&\leq \left(\frac{s}{r}\right)^{m}\left(\frac{(Ne)^{n'}}{n^n}\right)\frac{(\sum^{r}_{i=1}(n^{2}_ie)^{m_i})}{(2m)^{m}},
\end{align*}
where $m=\min\{m_1,\ldots,m_r\}$, $n=\min\{n_1,\ldots,n_r\}$, $n'=\max\{n_1,\ldots,n_r\}$.
By the assumption in Eq. $(\ref{eq2-1})$, we have 
$$
N <\left(\left(\frac{r}{s}\right)^{m}\frac{(2m)^{m})}{\sum^{r}_{i=1}(en^{2}_i)^{m_i}}\right)^{1/n'}\frac{n^{n/n'}}{e}, 
$$
which implies that 
$$
\Pr\left[\bigcup^{r}_{i=1}A_i\right]\leq \left(\frac{s}{r}\right)^{m}\left(\frac{(Ne)^{n'}}{n^n}\right)\frac{(\sum^{r}_{i=1}(n^{2}_ie)^{m_i})}{(2m)^{m}}<1.
$$
Therefore $\Pr[\bigcap^{r}_{i=1}\overline{A}]>0$, and there exists an $(s,t)$-coloring of $K_N$ such that there is no monochromatic copy of $G_i$ which all the edges receive the color containing $i$ where $1\leq i\leq r$.    
\end{proof}

\noindent {\bf Data availability:} Not applicable.

\noindent {\bf Code Availability:} Not applicable.

\section{Declarations}

{\bf Conflict of interest:} The authors declare that they have no conflict of interest.

\end{document}